\documentclass[11pt]{amsart}

\usepackage[english]{babel}
\usepackage[utf8]{inputenc}
\usepackage[T1]{fontenc}

\usepackage{amsfonts, amsmath, amssymb}
\usepackage{amsthm}
\usepackage{mathrsfs}

\usepackage{xcolor}

\newtheorem{thm}{Theorem}[section]
\newtheorem{prop}[thm]{Proposition}
\newtheorem{lemma}[thm]{Lemma}
\newtheorem{cor}[thm]{Corollary}

\theoremstyle{definition}
\newtheorem{defin}[thm]{Definition}
\newtheorem{rmk}[thm]{Remark}
\newtheorem{ex}[thm]{Example}

\DeclareMathOperator{\GL}{GL}

\newcommand{\kk}{\Bbbk}

\newcommand{\ZZ}{\mathbb{Z}}
\newcommand{\NN}{\mathbb{N}}

\newcommand{\A}{\mathbb{A}}

\newcommand{\m}{\mathfrak{m}}

\newcommand{\Hilb}[2]{\mathrm{Hilb}^{#1}(\A^{#2})}

\newcommand{\h}{\mathbf{h}}
\newcommand{\cH}{\mathscr{H}}
\newcommand{\cC}{\mathscr{C}}
\newcommand{\tC}{\widetilde{C}}
\newcommand{\tcC}{\widetilde{\mathscr{C}}}

\newcommand{\ComMat}[2]{C_{#1}(\mathbb{M}_{#2})}

\newcommand{\cA}{\mathcal{A}}

\newcommand{\II}{\mathbf{I}}

\begin{document}

\title{2-step ideals and commuting matrices}
\author{Klemen Šivic}
\address{University of Ljubljana, Faculty of mathematics and physics, Jadranska 19, 1000 Ljubljana, Slovenia, and Institute of mathematics, physics and mechanics, Jadranska 19, 1000 Ljubljana, Slovenia}
\email{klemen.sivic@fmf.uni-lj.si}
\maketitle

\begin{abstract}
2-step ideals are ideals $I$ of the polynomial ring that satisfy $\m^{k+2}\subsetneq I\subsetneq \m^k$ where $\m$ is the maximal ideal generated by variables. This class of ideals was recently introduced in [F. Giovenzana, L. Giovenzana, M. Graffeo, P. Lella, {\em New components of Hilbert schemes of points and 2-step ideals}, 2025, arxiv: 2507.02789] with the aim to obtain new irreducible components of $\Hilb{d}{n}$ and in particular to obtain new loci in $\Hilb{d}{3}$ of large dimension. In this paper we use the correspondence between Hilbert schemes and varieties of commuting matrices to define loci of commuting matrices that correspond to 2-step ideals. Then we estimate the dimensions of the obtained loci to get new proofs for the estimates of dimensions of loci of 2-step ideals in the case $n=3$. In this case we are also able to omit some assumptions in the dimension estimates in the above mentioned paper.
\end{abstract}

\section{Introduction}

The Hilbert scheme of points $\Hilb{d}{n}$ is a moduli space of zero-dimensional subschemes of $\A^n$ of degree $d$. Equivalently, it parametrizes ideals $I$ in the polynomial ring $\Bbbk[x_1,\ldots ,x_n]$ such that $\Bbbk[x_1,\ldots ,x_n]/I$ is $d$-dimensional vector space over $\Bbbk$. $\Hilb{d}{n}$ is irreducible and smooth for all $d$ if $n=2$ \cite{Fog}, but it is in general reducible and highly singular for $n\ge 3$, see e.g. \cite{Jel - generically nonreduced, Jel - pathologies}. The closure of the locus parametrizing $d$-tuples of distinct points is always an irreducible component of $\Hilb{d}{n}$ of dimension $dn$, called the smoothable component. For $n\ge 4$ this is the only component if and only if $d\le 7$ \cite{IE, Maz}. For various sufficiently large $d$ and $n$ many irreducible components of $\Hilb{d}{n}$ have been found \cite{CFM, GGGL, GGGL - unexpected, Hui1, Hui2, Iar,  IE,  IK, Jel - elementary components, Jel - generically nonreduced, KK, SS, SS2, Sha, Sza}, indicating that the number of components grows very quickly as $d$ and $n$ grow, and that the complete classification of components  is currently out of reach. For a fixed $n\ge 3$ the number of components is at least exponential in $d$ by \cite[Theorem 3]{Iar2}. The classification of components of $\Hilb{d}{n}$ was done only for $d\le 10$ \cite{CEVV, GKS}.

The most challenging case is $n=3$. There is a large gap where irreducibility of the Hilbert scheme is not known. It is known that $\Hilb{d}{3}$ is irreducible for $d\le 11$ \cite{DJNT}, while it is reducible for $d\ge 78$ by a result of Iarrobino \cite{Iar} that has not been improved in the last 42 years. Moreover, no irreducible component of a Hilbert scheme of points in $\A^3$, other than the smoothable one, is known, see \cite[Problem VII]{Jel - open problems}. Even more, not a single provably non-smoothable algebra is known. One difficulty seems to lie in the large tangent spaces, which indicate that all non-smoothable components of $\Hilb{d}{3}$ might be generically non-reduced, see \cite[Problem I]{Jel - open problems}, or \cite[p. 32]{GGGL}. The large tangent spaces in the case $n=3$ prevent to apply the most common techniques to find components of $\Hilb{d}{n}$, which include finding smooth points or showing that the negative tangents are trivial, the approach initiated by Jelisiejew in \cite{Jel - elementary components}.  Iarrobino's approach to prove reducibility of $\Hilb{78}{3}$ was nonconstructive. He found a locus of compressed algebras inside $\Hilb{78}{3}$ with dimension more than $3\cdot 78$. The dimension argument is still the only method that has been used to prove reducibility of $\Hilb{d}{3}$.

Motivated by the above questions the authors of \cite{GGGL} defined 2-step ideals as ideals $I$ of the polynomial ring $\Bbbk[x_1,\ldots ,x_n]$ lying between $(x_1,\ldots ,x_n)^{k+2}$ and $(x_1,\ldots ,x_n)^k$ for some $k>0$. If $I$ is a 2-step ideal, then the Hilbert function of the 2-step algebra $\Bbbk[x_1,\ldots ,x_n]/I$ is of the form $(r_0,r_1,\ldots ,r_{k-1},q_k,q_{k+1})$ where $r_j={j+n-1\choose n-1}$ for each $j$. Such a Hilbert function is called 2-step Hilbert function.  In \cite{GGGL} the authors estimated dimensions of the loci of 2-step ideals in some cases and applied these estimates to get various loci inside $\Hilb{d}{3}$ of dimension at least $3d$ if $d$ is large enough. They also showed that the loci of 2-step ideals of given Hilbert function are often irreducible components of $\Hilb{d}{n}$ for $n\ge 4$.

Hilbert schemes are closely related to varieties of commuting matrices, see \cite{GKS, HJ, JS, Nak}. The aim of this paper is to describe the subloci of varieties of triples of commuting matrices that correspond to the 2-step ideals of $\Bbbk[x_1,x_2,x_3]$ and to use this correspondence to give alternative proofs of some dimension estimates from \cite{GGGL}. Moreover, our technique will enable us to omit some assumptions in the results of \cite{GGGL}. More precisely, our main result is the following (see Corollary \ref{cor: main theorem}).

\begin{thm}\label{thm: main}
Let $n=3$, let $\h=(r_0,r_1,\ldots ,r_{k-1},q_k,q_{k+1})$ be a 2-step Hilbert function and let $H_\h$ be the sublocus of $\Hilb{d}{3}$ consisting of all local algebras with Hilbert function $\h$ that are supported at the origin. Then
$$\dim H_\h\ge q_kh_k+q_{k+1}(h_{k+1}-2h_k)$$
where $h_j=r_j-q_j$ for $j\in\{k,k+1\}$.
\end{thm}

The restriction to the case $n=3$ is motivated by the open problems on $\Hilb{d}{3}$ described above. In this case Theorem \ref{thm: main} is a generalization of \cite[Corollary 3.18]{GGGL}, which is one of the main dimension estimates in \cite{GGGL}. Note that in \cite[Corollary 3.18]{GGGL} it is assumed that $h_{k+1}\ge \frac{8}{3}h_k$, while our Theorem \ref{thm: main} holds without this assumption. We can therefore get new examples of loci in $\Hilb{d}{3}$ of dimension at least $3d$, however these examples all satisfy $d\ge 78$. The smallest example has $d=98$ and Hilbert function $(1,3,6,10,15,21,28,13,1)$, see Remark \ref{rmk: smallest new case}.

We briefly describe the techniques we use to prove Theorem \ref{thm: main}. To each 2-step Hilbert function $\h$ we assign a special locally closed set $\tC_\h$ of triples of commuting matrices, whose $\GL_d$-orbit corresponds to $H_\h$ in the sense of \cite{GKS, HJ, JS, Nak}. We estimate the dimension of $\tC_\h$ as follows. We define a projection $\tC_\h\to \tC_{\h'}$ where $\h'$ is obtained from $\h$ by deleting the last term. We observe that the image is a certain determinantal variety and the fibres are cokernels of certain matrices. We estimate the dimension of the image using well known estimates for codimensions of determinantal varieties. Theorem \ref{thm: main} then follows by the theorem on dimensions of fibres.

The paper is organized as follows. In Section 2 we recall the results from the literature that are needed later in the paper. We define 2-step ideals and describe the connection between Hilbert schemes and commuting matrices. We use this correspondence to define sets $\tC_\h$ in Section 3. In Section 4 we compute the dimension of the stabilizer of $\tC_\h$ in $\GL_d$. In Section 5 we compute the dimension of $H_\h$ if $\h$ is a 1-step Hilbert function, and in Section 6 we estimate the dimension for 2-step Hilbert functions.

\subsection*{Acknowledgments}

The author would like to thank Maciej Ga{\l}\k{a}zka, Michele Graffeo, Joachim Jelisiejew and Hanieh Keneshlou for many discussions on Hilbert schemes and commuting matrices. Many thanks also to two anonymous referees for careful reading of the paper and for their suggestions to improve the original manuscript.

The author is partially supported by the Slovenian Research and Innovation Agency program P1-0222 and grants J1-70017, J1-60011 and J1-50001.

\section{Preliminaries}

In this section we recall the notions of Hilbert function and of 2-step ideal, and describe the relation between Hilbert schemes and varieties of commuting matrices.

We first fix the notation. Throughout the paper we will work over algebraically closed field $\kk$. Let $R=\Bbbk[x_1,\ldots ,x_n]$ be the polynomial ring in $n$ variables, and $\m$ its maximal ideal generated by variables. In Section 2 we will work with an arbitrary $n$, while from Section 3 we will specialize to $n=3$.

\subsection{The Hilbert function and 2-step ideals}

We first recall the notion of Hilbert function. Let $\cA=\oplus_{j\in \ZZ}\cA_j$ be a $\ZZ$-graded $\Bbbk$-algebra of finite type and $M=\oplus_{j\in \ZZ}M_j$ a finitely generated $\ZZ$-graded $\cA$-module. The Hilbert function of $M$ is the function
\begin{eqnarray*}
\h_M\colon \ZZ&\to& \NN_0\\
j&\mapsto&\dim_{\Bbbk}M_j.
\end{eqnarray*}
In particular, the Hilbert function of a graded algebra $\cA$ is the Hilbert function of $\cA$ as a module over itself. If $(\cA,\mathfrak{n})$ is a local Artinian algebra, then the Hilbert function $\h_\cA$ of $\cA$ is defined to be the Hilbert function of its associated graded algebra
$$\mathrm{gr}\,  \cA=\oplus_{j=0}^{\infty}\mathfrak{n}^j/\mathfrak{n}^{j+1},$$
i.e.
$$\h_\cA(j)=\dim \mathfrak{n}^j-\dim \mathfrak{n}^{j+1},$$
for each $j$. We will always assume that $n$ is the embedding dimension of $\cA$, so $\h_\cA(1)=n$. We will use the notation from \cite{GGGL}. The $j$-th graded piece of $R$ will be denoted by $R_j$ and its dimension by $r_j$. Then
$$r_j={j+n-1\choose n-1},$$
for each $j$. If $R/I$ is a local or graded Artinian algebra, we define
$$q_j:=\h_{R/I}(j)\quad \mathrm{and}\quad h_j:=r_j-q_j,$$
for each $j$.

The Hilbert function of an Artinian algebra $\cA$ attains only finitely many nonzero values, so we usually write it as a vector of nonzero values $\h_\cA=(q_0,q_1,\ldots ,q_s)$ where $s$ is the so-called socle degree of $\cA$. The classification of vectors of positive integers that are Hilbert functions of Artinian algebras is given by Macaulay's theorem, see e.g. \cite[Theorem 4.2.10]{BH}.

Throughout the paper we will consider various loci of algebras having given Hilbert function, therefore we fix notation. Given a Hilbert function $\h$ we denote by $H_\h$ the sublocus of $\Hilb{d}{n}$ consisting of all local algebras with Hilbert function $\h$ that are supported at the origin, and by $\cH_\h$ the homogeneous sublocus of $H_\h$, i.e., the locus of all graded algebras with Hilbert function $\h$.

Following \cite{GGGL} we now define 1-step and 2-step ideals.

\begin{defin}
\begin{enumerate}
\item
An ideal $I$ of $R$ is {\em 1-step ideal of order $k>0$} if
$$\m^{k+1}\subsetneq I\subsetneq \m^k.$$
\item
An ideal $I$ of $R$ is {\em 2-step ideal of order $k>0$} if
$$\m^{k+2}\subsetneq I\subsetneq \m^k\quad \mathrm{and}\quad I\not\subseteq \m^{k+1}.$$
\item
A local Artinian algebra $\cA$ is {\em 2-step algebra} (resp. {\em 1-step algebra}) if it is of the form $\cA=R/I$ for some 2-step (resp. 1-step) ideal $I$.
\item
A Hilbert function $\h=(q_0,q_1,\ldots ,q_s)$ is {\em 2-step Hilbert function} (resp. {\em 1-step Hilbert function}) if it is the Hilbert function of some 2-step (resp. 1-step) algebra.
\end{enumerate}
\end{defin}

Note that every 2-step Hilbert function of order $k$ is of the form
$$\h=(r_0,r_1,\ldots ,r_{k-1},q_k,q_{k+1})$$
where $q_k$ is strictly smaller than $r_k$, and $q_{k+1}$ is nonzero if $\h$ is not 1-step Hilbert function. 1-step Hilbert functions are of the same form, but without the last term. Note also that a 2-step algebra of order 1 has embedding dimension smaller than $n$, therefore we will always assume that $k>1$.

\subsection{Hilbert schemes and commuting matrices}

In this subsection we recall the connection between Hilbert schemes and commuting matrices. We refer to \cite{GKS, HJ, JS, Nak}.

Let $\mathbb{M}_d$ be the set of all $d\times d$ matrices over $\Bbbk$. The identity matrix will be denoted by $\II$. Let
$$\ComMat{n}{d}=\{(A_1,\ldots ,A_n)\in \mathbb{M}_d^n;A_iA_j=A_jA_i\, \mathrm{for}\, \mathrm{all}\, i,j=1,\ldots ,n\}.$$
Although $\ComMat{n}{d}$ is traditionally called variety of $n$-tuples of commuting $d\times d$ matrices, we consider it as a subscheme of $\mathbb{M}_d^n$ defined by the ideal generated by the entries of all commutators $A_iA_j-A_jA_i$. Let $\mathcal{U}=\ComMat{n}{d}\times \kk^d$. We say that a point $(A_1,\ldots ,A_d,v)\in \mathcal{U}$ is stable if $v$ generates $\kk^d$ as a module over $\kk[A_1,\ldots ,A_n]$. In such case we say also that $v$ is a cyclic vector for the $n$-tuple $(A_1,\ldots ,A_n)\in \ComMat{n}{d}$. We denote the set of all stable $(n+1)$-tuples of $\mathcal{U}$ by $\mathcal{U}^{st}$. This is an open subscheme of $\mathcal{U}$.

If $(A_1,\ldots ,A_n)\in \ComMat{n}{d}$ is an arbitrary commuting $n$-tuple, we can define an $R$-module structure on $\Bbbk^d$ by $x_i\cdot u:=A_iu$. If $(A_1,\ldots ,A_d,v)\in \mathcal{U}^{st}$, then the map $f\mapsto f(A_1,\ldots ,A_d)v$ is an $R$-module homomorphism $R\to \kk^d$. This map is surjective, as the tuple $(A_1,\ldots ,A_d,v)$ is stable. Let $I$ be the kernel of this map. Then $\kk^d$ is isomorphic to $R/I$ as an $R$-module and $\kk[A_1,\ldots ,A_n]$ is isomorphic to $R/I$ as an algebra. Therefore, we have a map
\begin{eqnarray*}
p_1\colon \mathcal{U}^{st}&\to&\Hilb{d}{n}\\
(A_1,\ldots ,A_d,v)&\mapsto&R/I.
\end{eqnarray*}
This map defines a morphism of schemes and $\mathcal{U}^{st}$ is a principal $\GL_d$-bundle over $\Hilb{d}{n}$, see \cite[Proposition 3.7]{JS}. The converse construction (i.e., the construction of a section of the bundle) is as follows. Given an algebra $\cA=R/I\in \Hilb{d}{n}$ choose a $\kk$-basis of $\cA$ and, with respect to the chosen basis, let $A_i$ denote the matrix of multiplication with the variable $x_i$ on $\cA$. Then the matrices $A_1,\ldots ,A_d$ commute and $(A_1,\ldots ,A_d,1)$ is a stable point in $\mathcal{U}$ corresponding to $\cA$.

Apart from the above morphism $p_1\colon \mathcal{U}^{st}\to \Hilb{d}{n}$, we also have the projection $p_2\colon \mathcal{U}^{st}\to \ComMat{n}{d}$. Since $p_1$ defines a $\GL_d$-bundle, we have the following result.

\begin{lemma}[{\cite[\S 3.4]{JS}}]\label{lemma: bijection components of HS and commuting matrices}
The irreducible components of $\Hilb{d}{n}$ are in bijection with irreducible components of $\mathcal{U}^{st}$, which are in bijection with those irreducible components of $\ComMat{n}{d}$ that contain an $n$-tuple of matrices having a cyclic vector.

Moreover, if $\mathcal{Z}^H$ is an irreducible component of $\Hilb{d}{n}$ and $\mathcal{Z}^C$ is the corresponding irreducible component of $\ComMat{n}{d}$, then $\dim \mathcal{Z}^H=d-d^2+\dim \mathcal{Z}^C$.
\end{lemma}

The above result on dimensions generalizes to other locally closed subsets of $\Hilb{d}{n}$. In particular, if we define
$$C_\h=p_2(p_1^{-1}(H_\h))\quad \mathrm{and}\quad \cC_\h=p_2(p_1^{-1}(\cH_\h))$$
for a Hilbert function $\h$, then we have
\begin{equation}\label{eq: dim Hilbert - matrices}
\dim C_\h=d^2-d+\dim H_\h\quad \mathrm{and}\quad \dim \cC_\h=d^2-d+\dim \cH_\h.
\end{equation}

There is a relation between the support of the algebra in $\Hilb{d}{n}$ and eigenvalues of the corresponding commuting matrices, see \cite[\S 3.3 - 3.4]{JS}. Let $\cA\in \Hilb{d}{n}$ and let $(A_1,\ldots ,A_n)$ be the corresponding $n$-tuple of commuting matrices. If the maximal ideal $(x_1-\lambda_1,\ldots ,x_n-\lambda_n)$ is in the support of $\cA$, then $\lambda_i$ is an eigenvalue of $A_i$ for each $i$. Moreover, the support of $\cA$ contains only one point if and only if every matrix $A_i$ has only one eigenvalue. In particular, algebras supported at the origin correspond to $n$-tuples of nilpotent commuting matrices. Consequently, the elementary components of $\Hilb{d}{n}$ correspond to components of $\ComMat{n}{d}$ containing only $n$-tuples of matrices with a single eigenvalue. On the other hand, the smoothable component of $\Hilb{d}{n}$ corresponds to the principal component of $\ComMat{n}{d}$, that is the closure of the locus of $n$-tuples of simultaneously diagonalizable matrices. We note that the dimension of the principal component of $\ComMat{n}{d}$ is $d^2-d+nd$, which is a consequence of Lemma \ref{lemma: bijection components of HS and commuting matrices}, but it also follows directly from simultaneous diagonalization of matrices, see \cite[Proposition 6]{GS} or \cite[\S 3.4]{JS}.

There is also a close relation between the Hilbert function of a local Artinian algebra supported at the origin and the dimensions of the common cokernels of products of corresponding nilpotent commuting matrices. Before we describe this correspondence we recall some basic facts about cokernels of matrices. The cokernel of a matrix $A$ is the set of all  row vectors $v^T$ satisfying $v^TA=0$. It is canonically isomorphic to the kernel of the transpose of $A$. Note also that the common cokernel of $d\times d$ matrices $A_1,\ldots ,A_n$ is equal to the cokernel of the $d\times nd$ matrix $\begin{bmatrix}A_1 & A_2 & \cdots & A_n\end{bmatrix}$, and its dimension is equal to the codimension of the image of $\begin{bmatrix}A_1 & A_2 & \cdots & A_n\end{bmatrix}$ in $\Bbbk^d$.

Let $(\cA,\mathfrak{n})$ be a local Artinian algebra supported at the origin, and $(A_1,\ldots ,A_n)\in \ComMat{n}{d}$ the corresponding $n$-tuple of nilpotent commuting matrices. If $j$ is any positive integer, then
\begin{equation}\label{eq: cokernels and Hilbert function}
\dim_{\kk}\cA/\mathfrak{n}^j=\dim_{\kk}\bigcap_{i_1,\ldots ,i_j}\mathrm{coker}\, (A_{i_1}\cdots A_{i_j}),
\end{equation}
where the intersection runs over all $j$-tuples of (not necessarily different) indices.  We refer to \cite[Lemma 3.19 and Equations (3.20)]{JS}. Since $\dim_\Bbbk\cA/\mathfrak{n}=1$, it follows from \eqref{eq: cokernels and Hilbert function} that an $n$-tuple of nilpotent commuting  matrices has a cyclic vector (i.e., it belongs to $\mathcal{U}^{st}$) if and only if the common cokernel of the matrices is 1-dimensional.

\section{Commuting matrices corresponding to 2-step ideals}

In the rest of the paper we will assume that $n=3$. In particular, $r_j={j+2\choose 2}$ and
$$\sum_{i=0}^jr_i={j+3\choose 3}$$
for each $j$. Given any Hilbert function $\h$, in Section 2 we defined the subvariety $C_\h$ of $\ComMat{3}{d}$ that corresponds to the locus $H_\h\subseteq \Hilb{d}{3}$. In this section, for each 2-step Hilbert function $\h$, we will define a smaller subvariety of $\ComMat{3}{d}$ that corresponds to $H_\h$.

Let $I$ be a 2-step ideal of $R$ of order $k>1$, let $\cA=R/I$, let $\mathfrak{n}$ be the unique maximal ideal of $\cA$ and let $\h_\cA=(r_0,r_1,\ldots ,r_{k-1},q_k,q_{k+1})$ be the Hilbert function of $\cA$. Recall that the commuting matrices $A_1,A_2,A_3$ that correspond to $\cA\in \Hilb{d}{3}$ represent multiplications with the variables on $\cA$ in some basis. The aim of this section is to find a nice basis of $\cA$. The image of a polynomial $f$ under the projection $R\to \cA$ will be denoted by $\overline{f}$. First note that $\mathfrak{n}^{k+2}=\{0\}$ and that vector spaces $\mathfrak{n}^j/\mathfrak{n}^{j+1}$ and $\m^j/\m^{j+1}$ are isomorphic for $j=0,1,\ldots ,k-1$.

We choose a basis of $\cA$ as follows. First we choose linearly independent elements $f_1,\ldots ,f_{q_{k+1}}\in \mathfrak{n}^{k+1}$. Then we choose elements $g_1,\ldots ,g_{q_k}\in \mathfrak{n}^k$ that are linearly independent modulo $\mathfrak{n}^{k+1}$. Finally, for each $j=k-1,k-2,\ldots ,1,0$ we choose all the monomials of degree $j$ and order them lexicographically. Then $(f_1,\ldots ,f_{q_{k+1}},g_1,\ldots ,g_{q_k},\overline{x_1^{k-1}},\overline{x_1^{k-2}x_2},\ldots ,\overline{x_1},\overline{x_2},\overline{x_3},\overline{1})$ is a basis of $\cA$, where the last ${k+2\choose 3}$ elements are ordered with respect to graded lexicographic order. In this ordered basis the matrices $A_i$ are of the block form
\begin{equation}\label{eq: matrices A_i}
A_i=
\begin{bmatrix}
0 & B_i & C_i\\
& 0 & D_i\\
& & 0 & E_i^{(k-1)}\\
& & & \! \! \! \! \! \! \! 0 & E_i^{(k-2)}\\
& & & & \!\!\!\!\!\!\!\!\ddots & \ddots\\
& & & & & 0 & E_i^{(1)}\\
& & & & & & 0
\end{bmatrix}
\end{equation}
with rows and columns of respective sizes $q_{k+1},q_k,r_{k-1},\ldots ,r_1,r_0$, where all the non-specified entries are zeros. Moreover, the matrices $E_i^{(j)}$ are independent of the ideal $I$ and they are inductively defined as follows. For $j=1$ we have
$$E_1^{(1)}=\begin{bmatrix}1 \\ 0 \\ 0\end{bmatrix}, \quad 
E_2^{(1)}=\begin{bmatrix}0 \\ 1 \\ 0\end{bmatrix}, \quad
E_3^{(1)}=\begin{bmatrix}0 \\ 0 \\ 1\end{bmatrix}.$$
For $j=2,\ldots ,k-1$ write the matrix $E_i^{(j)}$ in the block form $E_i^{(j)}=\begin{bmatrix} F_i^{(j)} \\ G_i^{(j)}\end{bmatrix}$ where blocks have $r_{j-1}$ resp. $r_j-r_{j-1}=j+1$ rows. Then $F_1^{(j)}=\II$, $G_1^{(j)}=0$, $F_i^{(j)}=\begin{bmatrix}E_i^{(j-1)} & 0 \end{bmatrix}$ for $i=2,3$ and
\begin{equation}\label{eq: G_i}
G_2^{(j)}=\begin{bmatrix}
0 & \cdots & 0 & 1\\
\vdots & & & \ddots & \ddots\\
0 & \cdots & & \cdots & 0 & 1\\
0 & \cdots & & & \cdots & 0
\end{bmatrix} \quad \mathrm{and}\quad G_3^{(j)}=\begin{bmatrix}
0 & \cdots & 0 & 0 & \cdots & 0\\
0 & \cdots & 0 & 1 & \ddots & \vdots \\
\vdots & &  & \ddots & \ddots & 0\\
0 & \cdots & & \cdots & 0 & 1
\end{bmatrix}.
\end{equation}

Since $(r_0,r_1,\ldots ,r_{k-1},q_k,q_{k+1})$ is the Hilbert function of $\cA$, equation \eqref{eq: cokernels and Hilbert function} implies that
\begin{equation}\label{eq: dim cokernels 1}
\dim_{\kk}\bigcap_{i_1,\ldots ,i_j}\mathrm{coker}\, (A_{i_1}\cdots A_{i_j})=r_0+r_1+\cdots +r_{j-1}\quad \mathrm{for}\, \, j=1,2,\ldots ,k
\end{equation}
and
\begin{equation}\label{eq: dim cokernels 2}
\dim_{\kk}\bigcap_{i_1,\ldots ,i_{k+1}}\mathrm{coker}\, (A_{i_1}\cdots A_{i_{k+1}})=r_0+r_1+\cdots +r_{k-1}+q_k.
\end{equation}
Recall from section 2 that the common cokernel of $A_1,A_2,A_3$ is 1-dimensional. Therefore, using the block structure of the matrices $A_i$ and the equation \eqref{eq: dim cokernels 1} for $j=1$ we get that the cokernels of $\begin{bmatrix}B_1 & B_2 & B_3\end{bmatrix}$ and of $\begin{bmatrix}D_1 & D_2 & D_3\end{bmatrix}$ are trivial. Note that the cokernel of $\begin{bmatrix}E_1^{(j)} & E_2^{(j)} & E_3^{(j)}\end{bmatrix}$ is trivial for $j=1,\ldots ,k-1$ by the definition of the matrices $E_i^{(j)}$. Moreover, the vector $e_d:=\begin{bmatrix} 0 \\ \vdots \\ 0 \\ 1\end{bmatrix}$ does not belong to the linear span of the columns of the matrix $\begin{bmatrix} A_1 & A_2 & A_3\end{bmatrix}$ and it is cyclic vector for the triple $(A_1,A_2,A_3)$.

We note two more things. First, the ideal $I$ is homogeneous if and only if $C_1=C_2=C_3=0$. Second, when $I$ is a 1-step ideal of order $k$, we get matrices of the same shape as \eqref{eq: matrices A_i}, but without the first block row and column.

Given a 2-step Hilbert function $\h$, we now define $\tC_\h$ to be the set of all triples of commuting matrices $(A_1,A_2,A_3)$ of the shape \eqref{eq: matrices A_i} with $\mathrm{coker}\, \begin{bmatrix}B_1 & B_2 & B_3\end{bmatrix}=\{0\}$ and $\mathrm{coker}\, \begin{bmatrix}D_1 & D_2 & D_3\end{bmatrix}=\{0\}$. Furthermore, let $\tcC_\h$ be the subset of $\tC_\h$ consisting of all triples with $C_1=C_2=C_3=0$. We note that the above definition defines $\tC_\h$ also for 1-step Hilbert functions if we delete the first block row and column in \eqref{eq: matrices A_i}.

The group $\GL_d$ acts on $C_\h$ by simultaneous conjugation of matrices. For each 2-step Hilbert function $\h$ of order $k>1$ we proved above that each algebra $\cA\in H_\h$ corresponds to a triple from $\tC_\h$. The converse is also true: If $(A_1,A_2,A_3)\in \tC_\h$, then the common cokernel of $A_1, A_2 ,A_3$ is 1-dimensional, since the common cokernels of $B_1,B_2,B_3$ and of $D_1,D_2,D_3$ are trivial. Therefore, the triple $(A_1,A_2,A_3)$ has a cyclic vector, so it corresponds to an algebra $\cA\in \Hilb{d}{3}$. Using the block structure of the matrices $A_i$ and the conditions on the common cokernels of $B_1,B_2,B_3$ and of $D_1,D_2,D_3$ it is easy to show (for example, by induction on $j$) that the equalities \eqref{eq: dim cokernels 1} and \eqref{eq: dim cokernels 2} are satisfied. The equation \eqref{eq: cokernels and Hilbert function} then implies that the Hilbert function of $\cA$ is $(r_0,r_1,\ldots ,r_{k-1},q_k,q_{k+1})$, i.e. $\cA\in H_\h$. We get the following.

\begin{lemma}
$C_\h$ and $\cC_\h$ are $\GL_d$-orbits of $\tC_\h$ and $\tcC_\h$, respectively. 
\end{lemma}

\section{The stabilizer of $\tC_\h$}

Throughout this section let $\h=(r_0,r_1,\ldots ,r_{k-1},q_k,q_{k+1})$ be a 2-step Hilbert function of order $k>1$. (Note that $q_{k+1}$ may be zero, in the case when $\h$ is a 1-step Hilbert function.) Since $C_\h$ is the $\GL_d$-orbit of $\tC_\h$, to relate the dimensions of $C_\h$ and $\tC_\h$, we have to analyse the stabilizer of $\tC_\h$ in $\GL_d$. However, the stabilizer has a complicated structure due to the fact that for every triple $(A_1,A_2,A_3)\in \tC_\h$ we have to take into account the invertible matrices from the common centralizer of $A_1$, $A_2$ and $A_3$. Therefore we consider the action of a smaller group.

Let $G$ be the group of all $d\times d$ invertible matrices with the last column equal to $e_d$.

\begin{prop}\label{prop: G-orbit dense}
The $G$-orbit of $\tC_\h$ is dense in $C_\h$.
\end{prop}

\begin{proof}
We will show that for every $P\in \GL_d$ with a non-zero $(d,d)$-th entry and every triple $(A_1,A_2,A_3)\in \tC_\h$ there exists $Q\in G$ such that $P^{-1}A_iP=Q^{-1}A_iQ$ for $i=1,2,3$. Since $C_\h$ is the $\GL_d$-orbit of $\tC_\h$ and since the set of all matrices with a non-zero $(d,d)$-th entry is dense in $\GL_d$, this will prove the lemma.

Recall that the vector $e_d$ is a cyclic vector for the triple $(A_1,A_2,A_3)$. Therefore there exists a matrix $R$ in the algebra $\kk[A_1,A_2,A_3]$ whose last column is equal to the last column of $P$. Since $A_1$, $A_2$ and $A_3$ are strictly upper triangular, the matrix $R$ is invertible if and only if the $(d,d)$-th entry of $P$ is nonzero. In such case define $Q=R^{-1}P$. Since the last columns of $P$ and $R$ are equal, $Q$ belongs to $G$. Moreover, the fact that $R=PQ^{-1}$ belongs to the common centralizer of $A_1$, $A_2$ and $A_3$ implies $P^{-1}A_iP=Q^{-1}A_iQ$ for $i=1,2,3$.
\end{proof}

The above proposition immediately implies the following.

\begin{cor}\label{cor: dimension of G-orbit}
The dimension of $C_\h$ is equal to the dimension of the $G$-orbit of $\tC_\h$.
\end{cor}

We can now relate the dimension of $C_\h$ to that of $\tC_\h$.

\begin{prop}\label{prop: dim of stabilizer}
Let $\h=(r_0,r_1,\ldots ,r_{k-1},q_k,q_{k+1})$ be a 2-step Hilbert function of order $k>1$ and $d={k+2\choose 3}+q_k+q_{k+1}$. Then
$$\dim C_\h=\dim \tC_\h+d^2-d-q_{k+1}^2-q_kq_{k+1}-q_k^2.$$
\end{prop}

\begin{proof}
We define a map
\begin{eqnarray*}
\phi\colon G\times \tC_\h&\to& C_ \h\\
(P,A_1,A_2,A_3)&\mapsto& (P^{-1}A_1P,P^{-1}A_2P,P^{-1}A_3P).
\end{eqnarray*}
The map is dominant by Proposition \ref{prop: G-orbit dense} and it is clear that $\dim G=d^2-d$, so we have to prove that the fibres have dimension $q_{k+1}^2+q_kq_{k+1}+q_k^2$.

Due to homogeneity of the group action it suffices to consider the fibres where $P=\II$. So, assume that $\phi(\II,A_1,A_2,A_3)=\phi(Q,A_1',A_2',A_3')$ for some $(A_1,A_2,A_3),(A_1',A_2',A_3')\in \tC_\h$ and some $Q\in G$. With respect to the block partition of \eqref{eq: matrices A_i} write $Q$ in the block form
$$Q=\begin{bmatrix}
Q_{k+1,k+1} & Q_{k+1,k} & \cdots & Q_{k+1,0}\\
Q_{k,k+1} & Q_{k,k} & \cdots & Q_{k,0}\\
\vdots & \vdots & \ddots & \vdots\\
Q_{0,k+1} & Q_{0,k} & \cdots & Q_{0,0}
\end{bmatrix}.$$
(The indices in the above matrix may look unusual, but they  are labelled in such a way that $Q_{i,j}$ has $r_i$ resp. $q_i$ rows and $r_j$ resp. $q_j$ columns.)
First note that by the definition of $G$ we have $Q_{0,0}=1$ and $Q_{j,0}=0$ for $j=1,\ldots ,k+1$.

To get further conditions on $Q$, consider the common cokernels of the matrices $A_i$ and $A_i'$. Since $A_i=Q^{-1}A_i'Q$ for $i=1,2,3$, for all $j=1,\ldots ,k+1$ we get
\begin{equation}\label{eq: Q cokernels}
\bigcap_{i_1,\ldots ,i_j}\mathrm{coker}\, (A_{i_1}\cdots A_{i_j})=\left(\bigcap_{i_1,\ldots ,i_j}\mathrm{coker}\, (A_{i_1}'\cdots A_{i_j}')\right)\cdot Q
\end{equation}
where the intersections run over all $j$-tuples of indices. By the shape of the matrices $A_i$ and $A_i'$ both the above intersections of cokernels contain all block row vectors having arbitrary last $j$ blocks and zeros elsewhere. Moreover, using \eqref{eq: dim cokernels 1}, if $j\le k$, or \eqref{eq: dim cokernels 2}, if $j=k+1$, both the above intersections consist of exactly such row vectors. It follows from \eqref{eq: Q cokernels} that multiplication with $Q$ (from the right) preserves these spaces for $j=1,\ldots ,k+1$, therefore $Q$ is block upper triangular, i.e., $Q_{i,j}=0$ if $i<j$.

The next step is to show that $Q_{j,j}=\II$ for $j=0,1,\ldots ,k-1$ and $Q_{j',j}=0$ if $0\le j\le k-1$ and $j<j'\le k+1$. We do this by induction on $j$. The case $j=0$ is done by the definition of $G$. To perform inductive step suppose that the claim holds for some $j\in\{0,1,\ldots ,k-2\}$. For all $i=1,2,3$ we have $QA_i=A_i'Q$, and by the inductive assumption in the $(k+2-j)$-th column we get equations
$$Q_{j+1,j+1}E_i^{(j+1)}=E_i^{(j+1)}\quad \mathrm{and}\quad Q_{j',j+1}E_i^{(j+1)}=0\quad \mathrm{for}\, \, j'=j+2,\ldots ,k+1,$$
for $i=1,2,3$. The common cokernel of $E_1^{(j+1)}$, $E_2^{(j+1)}$ and $E_3^{(j+1)}$ is trivial, therefore we get $Q_{j+1,j+1}=\II$ and $Q_{j',j+1}=0$ for $j'=j+2,\ldots ,k+1$, completing the inductive step.

We proved that
\begin{equation}\label{eq: final Q}
Q=\begin{bmatrix}
Q_{k+1,k+1} & Q_{k+1,k} & 0 & & \cdots & 0\\
0 & Q_{k,k} & 0 & & \cdots & 0\\
0 & 0 & \II & 0 & \cdots & 0\\
\vdots & \vdots & \ddots & \ddots & \ddots & \vdots \\
0 & 0 & \cdots & 0 & \II & 0\\
0 & 0 & \cdots & & 0 & 1
\end{bmatrix}.
\end{equation}
It is easy to see that such a matrix belongs to the stabilizer of $\tC_\h$, so the triple $(QA_1Q^{-1},QA_2Q^{-1},QA_3Q^{-1})$ indeed belongs to $\tC_\h$. It follows that the fibre $\phi^{-1}(\phi(\II,A_1,A_2,A_3))$ consists of all quadruples $(Q,QA_1Q^{-1},QA_2Q^{-1},QA_3Q^{-1})$ where $Q$ is invertible matrix of the form \eqref{eq: final Q}, and hence the dimension of each fibre is $q_{k+1}^2+q_kq_{k+1}+q_k^2$. This concludes the proof of the proposition.
\end{proof}

\begin{rmk}\label{rmk: dim of stabilizer 1-step}
The above proposition in the particular case when $q_{k+1}=0$ implies that if $\h=(r_0,r_1,\ldots ,r_{k-1},q_k)$ is a 1-step Hilbert function of order $k>1$ and $d={k+2\choose 3}+q_k$, then
$$\dim C_\h=\dim \tC_\h+d^2-d-q_k^2.$$
\end{rmk}

\section{Dimension of the locus of 1-step ideals}

If $\h=(r_0,r_1,\ldots ,r_{k-1},q_k)$ is a 1-step Hilbert function, then it is clear (and well known) that $H_\h$ is isomorphic to the Grassmannian $\mathrm{Gr}(h_k,r_k)$, so it is irreducible and of dimension $h_kq_k$. However, the statement about irreducibility and dimension follows also directly from Lemma \ref{lemma: commutativity 1 - vector spaces} below. This lemma will be needed also in the next section, therefore we also include the proof of irreducibility of $H_\h$ and compute its dimension.

Note that in the definition of $\tC_\h$ the matrices $E_i^{(j)}$ are fixed, and in the case when $\h$ is a 1-step Hilbert function the matrices $A_i$ of the shape \eqref{eq: matrices A_i} do not have the first block row and column, so in this case the set $\tC_\h$ is isomorphic to
\begin{equation}\label{eq: D_iE_j=D_jE_i}
\left\{(D_1,D_2,D_3)\in \mathbb{M}_{q_k\times r_{k-1}}^3;D_iE_j^{(k-1)}=D_jE_i^{(k-1)}\, \mathrm{for}\, i,j=1,2,3\, \mathrm{and}\, \mathrm{coker}\begin{bmatrix} D_1 & D_2 & D_3 \end{bmatrix} =\{0\}\right\}.
\end{equation}
In this and the next section, to simplify the notation, we denote $E_i:=E_i^{(k-1)}$, $F_i:=F_i^{(k-1)}$ and $G_i:=G_i^{(k-1)}$ for $i=1,2,3$. This will not make any confusion as the matrices $E_i^{(k-1)}$ are the only ones among $E_i^{(j)}$ imposing commuting conditions. Recall that $E_i=\begin{bmatrix} F_i \\ G_i \end{bmatrix}$ for each $i$, with blocks of respective sizes $r_{k-2}\times r_{k-2}$ and $k\times r_{k-2}$, and that $F_1=\II$ and $G_1=0$.

We first show the following.

\begin{lemma}\label{lemma: commutativity of E_i}
$E_2F_3=E_3F_2.$
\end{lemma}

\begin{proof}
For $i=2,3$ the matrix $F_i$ represents the linear map
\begin{eqnarray*}
R_{k-2}&\to&R_{k-2}\\
x_1f(x_1,x_2,x_3)+g(x_2,x_3)&\mapsto&x_if(x_1,x_2,x_3).
\end{eqnarray*}
Since $E_i$ represents multiplication with $x_i$, the equality $E_2F_3=E_3F_2$ follows.
\end{proof}

The following lemma is the key technical result needed to compute $\dim \tC_ \h$ in the case of 1-step Hilbert function and the difference $\dim \tC_\h-\dim \tcC_\h$ in the case of 2-step Hilbert function. It is also one of the main tools needed to estimate $\dim \tcC_\h$ if $\h$ is a 2-step Hilbert function.

\begin{lemma}\label{lemma: commutativity 1 - vector spaces}
Let $t$ be a positive integer. Then the set
\begin{equation}\label{eq: linear system}
\{(X_1,X_2,X_3)\in \mathbb{M}_{t\times r_{k-1}}^3;X_iE_j=X_jE_i\, \mathrm{for}\, i,j=1,2,3\}
\end{equation}
is a vector space of dimension $tr_k$.
\end{lemma}

\begin{proof}
We only have to compute the dimension. As above, write
$$E_1=\begin{bmatrix} \II \\  0
\end{bmatrix}\quad \mathrm{and}\quad E_i=\begin{bmatrix}
F_i \\ G_i
\end{bmatrix}\, \, \mathrm{for}\, i=2,3,$$
where the blocks have $r_{k-2}$ respectively $k$ rows. With respect to the same block partitions write $X_i=\begin{bmatrix}Y_i & Z_i\end{bmatrix}$ for $i=1,2,3$. Equations $X_iE_1=X_1E_i$ for $i=2,3$ are then equivalent to
$$Y_i=Y_1F_i+Z_1G_i.$$
Equation $X_2E_3=X_3E_2$ is then equivalent to
$$Y_1F_2F_3+Z_1G_2F_3+Z_2G_3=Y_1F_3F_2+Z_1G_3F_2+Z_3G_2,$$
which is by Lemma \ref{lemma: commutativity of E_i} equivalent to $Z_2G_3=Z_3G_2$. The matrices $G_2$ and $G_3$ are explicitly given in \eqref{eq: G_i}, and we see that the space of solutions of the equation $Z_2G_3=Z_3G_2$ has dimension $t(k+1)$. Since $X_1$ was arbitrary, the space of solutions of \eqref{eq: linear system} has dimension $tr_{k-1}+t(k+1)=tr_k.$
\end{proof}

\begin{cor}\label{cor: dim 1-step matrices}
Let $\h=(r_0,r_1,\ldots ,r_{k-1},q_k)$ be a 1-step Hilbert function. Then $\tC_\h$ is irreducible and of dimension $q_kr_k$.
\end{cor}

\begin{proof}
Lemma \ref{lemma: commutativity 1 - vector spaces} implies that the set of all triples $(D_1,D_2,D_3)\in \mathbb{M}_{q_k\times r_{k-1}}^3$ satisfying $D_iE_j=D_jE_i$ for all $i$ and $j$ is a vector space of dimension $q_kr_k$. By \eqref{eq: D_iE_j=D_jE_i}, we need to show that this vector space contains triples where the rows of the matrix $\begin{bmatrix}D_1 & D_2 & D_3\end{bmatrix}$ are linearly independent, see Section 2. By the proof of Lemma \ref{lemma: commutativity 1 - vector spaces} we see that the rows of $\begin{bmatrix}D_1 & D_2 & D_3\end{bmatrix}$ belong to a $r_k$-dimensional vector space. Since $q_k<r_k$, the corollary follows.
\end{proof}

The following corollary now immediately follows from Proposition \ref{prop: dim of stabilizer} and equations \eqref{eq: dim Hilbert - matrices}.

\begin{cor}\label{cor: dim 1-step orbit, ideals}
Let $\h=(r_0,r_1,\ldots ,r_{k-1},q_k)$ be a 1-step Hilbert function and $d={k+2\choose 3}+q_k$. Then we have the following.
\begin{enumerate}
\item
$C_\h$ is irreducible and of dimension $d^2-d+q_kh_k.$
\item
$H_\h$ is irreducible and of dimension $q_kh_k$.
\end{enumerate}
\end{cor}

\section{Dimension of the locus of 2-step ideals}

In this section let $\h=(r_0,r_1,\ldots ,r_{k-1},q_k,q_{k+1})$ be a 2-step Hilbert function. We will assume that $q_{k+1}\ge 1$, as the case $q_{k+1}=0$ was considered in the previous section. We start with the following observation that follows immediately from Lemma \ref{lemma: commutativity 1 - vector spaces}.

\begin{lemma}\label{lemma: vector bundle over homogeneous locus}
The set $\tC_\h$ is a vector bundle over $\tcC_ \h$ with fibres of dimension $q_{k+1}r_k$.
\end{lemma}

Therefore, in the rest of this section, it suffices to consider the set $\tcC_\h$, whose $\GL_d$-orbit corresponds to the homogeneous locus $\cH_\h$.  Note that $\tcC_\h$ is isomorphic to the set of all 6-tuples $(B_1,B_2,B_3,D_1,D_2,D_3)\in \mathbb{M}_{q_{k+1}\times q_k}^3\times \mathbb{M}_{q_k\times r_{k-1}}^3$ satisfying
$$B_iD_j=B_jD_i\quad \mathrm{and}\quad D_iE_j=D_jE_i\, \, \mathrm{for}\, i,j=1,2,3,$$
and
$$\mathrm{coker}\begin{bmatrix} B_1 & B_2 & B_3 \end{bmatrix} =\{0\}\quad \mathrm{and}\quad \mathrm{coker}\begin{bmatrix} D_1 & D_2 & D_3 \end{bmatrix} =\{0\}.$$
We denote $\h'=(r_0,r_1,\ldots ,r_{k-1},q_k)$ and consider the projection
\begin{eqnarray*}
\pi\colon \tcC_\h&\to&\tC_{\h'}\\
(B_1,B_2,B_3,D_1,D_2,D_3)&\mapsto&(D_1,D_2,D_3).
\end{eqnarray*}

\begin{lemma}\label{lemma: fibre = cokernel}
The fibre $\pi^{-1}(D_1,D_2,D_3)$ consists of all 6-tuples where the rows of the matrix $\begin{bmatrix}B_1 & B_2 & B_3\end{bmatrix}$ are linearly independent and belong to the cokernel of the matrix
\begin{equation}\label{eq: matrix of the system for B_i}
\begin{bmatrix}
0 & D_3 & -D_2\\
-D_3 & 0 & D_1\\
D_2 & -D_1 & 0\\
\end{bmatrix}.
\end{equation}
In particular, the fibre is nonempty if and only if the dimension of the cokernel of \eqref{eq: matrix of the system for B_i} is at least $q_{k+1}$.
\end{lemma}

\begin{proof}
The lemma follows immediately from the equations $B_iD_j=B_jD_i$ and the fact that the common cokernel of $B_1,B_2,B_3$ is trivial.
\end{proof}

We now define a smaller matrix than \eqref{eq: matrix of the system for B_i} which will define an equivalent system of equations. For an arbitrary triple $(D_1,D_2,D_3)\in \tC_{\h'}$ write $D_i=\begin{bmatrix}D_i' & D_i''\end{bmatrix}$ for $i=1,2,3$, where the blocks are of respective sizes $q_k\times r_{k-2}$ and $q_k\times k$, and define
$$M_{D_1,D_2,D_3}=\begin{bmatrix}
0 & D_3 & -D_2\\
-D_3'' & 0 & D_1\\
D_2'' & -D_1 & 0
\end{bmatrix}.$$

\begin{lemma}\label{lemma: smaller matrix}
Let $(D_1,D_2,D_3)\in \tC_{\h'}$ be arbitrary. Then the matrices \eqref{eq: matrix of the system for B_i} and $M_{D_1,D_2,D_3}$ have the same cokernel.
\end{lemma}

\begin{proof}
As in the proof of Lemma \ref{lemma: commutativity 1 - vector spaces} we write $E_1=\begin{bmatrix}\II \\ 0 \end{bmatrix}$ and compute $D_i'=D_1E_i$ for $i=2,3$. We split the first block of \eqref{eq: matrix of the system for B_i} and compute
$$\begin{bmatrix}
0 & 0 & D_3 & -D_2\\
-D_3' & -D_3'' & 0 & D_1\\
D_2' & D_2'' & -D_1 & 0
\end{bmatrix}\cdot \begin{bmatrix} 
I & 0 & 0 & 0\\
0 & I & 0 & 0\\
E_2 & 0 & I & 0\\
E_3 & 0 & 0 & I
\end{bmatrix}=\begin{bmatrix}
0 & 0 & D_3 & -D_2\\
0 & -D_3'' & 0 & D_1\\
0 & D_2'' & -D_1 & 0
\end{bmatrix},$$
where we used $D_i'=D_1E_i$ for $i=2,3$ and $D_2E_3=D_3E_2$. The lemma now follows.
\end{proof}

We can now estimate the dimension of $\tcC_\h$.

\begin{thm}\label{thm: dim homogeneous C_h}
Let $\h=(r_0,r_1,\ldots ,r_{k-1},q_k,q_{k+1})$ be a 2-step Hilbert function. Then
$$\dim \tcC_\h\ge q_kr_k+q_{k+1}(r_{k+1}-3h_k).$$
\end{thm}

\begin{proof}
As above, let $\h'=(r_0,r_1,\ldots ,r_{k-1},q_k)$ and let $\pi\colon \tcC_\h\to \tC_{\h'}$ be the projection sending $(B_1,B_2,B_3,D_1,D_2,D_3)$ to $(D_1,D_2,D_3)$. We consider the fibres of this map. By Lemmas \ref{lemma: fibre = cokernel} and \ref{lemma: smaller matrix} the fibre $\pi^{-1}(D_1,D_2,D_3)$ is nonempty if and only if $M_{D_1,D_2,D_3}$ has at least $q_{k+1}$-dimensional cokernel, which holds if and only if
$$\mathrm{rank}\, M_{D_1,D_2,D_3}\le 3q_k-q_{k+1}=k(k+2)+h_{k+1}-3h_k.$$

We first consider the easy case when $h_{k+1}-3h_k\ge 0$. The authors of \cite{GGGL} refer to this case as the case "with no linear syzygies". The matrix $M_{D_1,D_2,D_3}$ has $2r_{k-1}+k=k(k+2)$ columns, so its rank is always at most $k(k+2)\le k(k+2)+h_{k+1}-3h_k$. Therefore the projection $\pi$ is surjective. Moreover, for each $(D_1,D_2,D_3)\in \tC_{\h'}$ we have
$$\dim \pi^{-1}(D_1,D_2,D_3)=q_{k+1}\cdot \dim \mathrm{coker}\, M_{D_1,D_2,D_3}\ge q_{k+1}(3q_k-k(k+2)),$$
so the theorem on dimensions of fibres (see e.g. \cite[Theorem 11.12]{Har}) together with Corollary \ref{cor: dim 1-step matrices} implies
$$\dim \tcC_\h\ge \dim \tC_{\h'}+q_{k+1}(3q_k-k(k+2))=q_kr_k+q_{k+1}(r_{k+1}-3h_k),$$
as required.

Now we consider the case when $h_{k+1}-3h_k<0$. In this case the image $\pi(\tcC_\h)$ is the intersection of the determinantal variety $\mathcal{Z}$ consisting of all triples $(D_1,D_2,D_3)\in \mathbb{M}_{q_k\times r_{k-1}}^3$ satisfying
$$D_iE_j=D_jE_i\, \mathrm{for}\, i,j=1,2,3\quad \mathrm{and}\quad \mathrm{rank}\, M_{D_1,D_2,D_3}\le k(k+2)+h_{k+1}-3h_k=3q_k-q_{k+1}$$
and the open set of all triples $(D_1,D_2,D_3)$ with $\mathrm{coker}\, \begin{bmatrix}D_1 & D_2 & D_3\end{bmatrix}=\{0\}$. Note that $\h$ is a Hilbert function of some (2-step) Artinian algebra, therefore $\tcC_\h$ is nonempty, hence the image $\pi(\tcC_\h)$ is some nonempty open subset of $\mathcal{Z}$. Using known bounds on codimensions of determinantal varieties (see e.g. \cite[Exercise 10.9]{Eis}) and Lemma \ref{lemma: commutativity 1 - vector spaces} we get that each irreducible component of $\mathcal{Z}$ has dimension at least
$$q_kr_k+q_{k+1}(h_{k+1}-3h_k).$$
Since $\pi(\tcC_\h)$ is nonempty and open in $\mathcal{Z}$, its closure is a union of some irreducible components of $\mathcal{Z}$, therefore
$$\dim \pi(\tcC_\h)\ge q_kr_k+q_{k+1}(h_{k+1}-3h_k).$$
Moreover, for each $(D_1,D_2,D_3)\in \pi(\tcC_\h)$ we have
$$\dim \pi^{-1}(D_1,D_2,D_3)=q_{k+1}\cdot \dim \mathrm{coker}\, M_{D_1,D_2,D_3}\ge q_{k+1}^2,$$
and the theorem follows from the theorem on dimensions of fibres.
\end{proof}

Now we can prove Theorem \ref{thm: main} from the introduction, here Corollary \ref{cor: main theorem}. Namely, Theorem \ref{thm: dim homogeneous C_h}, Lemma \ref{lemma: vector bundle over homogeneous locus}, Proposition \ref{prop: dim of stabilizer} and equations \eqref{eq: dim Hilbert - matrices} immediately imply the following.

\begin{cor}\label{cor: main theorem}
Let $\h=(r_0,r_1,\ldots ,r_{k-1},q_k,q_{k+1})$ be a 2-step Hilbert function and $d={k+2\choose 3}+q_k+q_{k+1}$. Then we have the following.
\begin{enumerate}
\item
$\dim \tC_\h\ge q_kr_k+q_{k+1}(r_k+r_{k+1}-3h_k),$
\item
$\dim C_\h\ge d ^2-d+q_kh_k+q_{k+1}(h_{k+1}-2h_k),$
\item
$\dim H_\h\ge q_kh_k+q_{k+1}(h_{k+1}-2h_k).$
\end{enumerate}
\end{cor}

\begin{rmk}\label{rmk: smallest new case}
Taking into account different supports of the local algebras, Theorem \ref{thm: main} gives the following estimate for the dimension of the Hilbert scheme:
$$\dim \Hilb{d}{3}\ge q_kh_k+q_{k+1}(h_{k+1}-2h_k)+3.$$
This estimate enables us to find loci inside $\Hilb{d}{3}$ of dimension larger than or equal to $3d$. Since Theorem \ref{thm: main} is a generalization of the case $n=3$ of \cite[Corolary 3.18]{GGGL} relaxing the hypothesis $h_{k+1}\ge \frac{8}{3}h_k$, we get new large loci inside $\Hilb{d}{3}$ for sufficiently large $d$. However, we always get $d\ge 78$, and the smallest new case with $h_{k+1}<\frac{8}{3}h_k$ has $d=98$ and Hilbert function $(1,3,6,10,15,21,28,13,1)$.
\end{rmk}

\begin{rmk}
One may try to get sharper estimates for $\dim H_\h$ by looking at smaller determinantal subvarieties of $\pi(\tcC_\h)$ where the fibres have larger dimension. However, using the classical bounds for codimensions of determinantal varieties we would get weaker estimates than in Theorem \ref{thm: main}. On the other hand, matrices $M_{D_1,D_2,D_3}$ have a special structure, and it is highly possible to get stronger estimates by exploiting the structure of these matrices. In fact, the next example shows that the estimate in Theorem \ref{thm: main} can indeed be significantly improved if the difference $3h_k-h_{k+1}$ is positive and sufficiently large. For this reason, the problem requires further investigation.
\end{rmk}

\begin{ex}
Consider a 2-step Hilbert function of the form $\h=(r_0,r_1,\ldots ,r_{k-1},1,1)$. Theorem \ref{thm: main} then gives $\dim H_\h\ge h_{k+1}-h_k=k+2$.

On the other hand, $\tcC_\h$ is the set of all 6-tuples $(B_1,B_2,B_3,D_1,D_2,D_3)\in \Bbbk^3\times \mathbb{M}_{1\times r_{k-1}}^3$ satisfying $B_iD_j=B_jD_i$ and $D_iE_j=D_jE_i$ for all $i$ and $j$, such that at least one of $B_1,B_2,B_3$ is nonzero and at least one of $D_1,D_2,D_3$ is nonzero. Up to $\GL_3$-action we may assume that $B_1=1$ and $B_2=B_3=0$, which gives $D_2=D_3=0$ and $D_1E_2=D_1E_3=0$. Using the description of the matrices $E_i^{(j)}$ given in Section 3, it is easy to see that the common cokernel of $E_2$ and $E_3$ is 1-dimensional, so $\dim \tcC_\h=3+1=4$. Using Lemma \ref{lemma: vector bundle over homogeneous locus}, Proposition \ref{prop: dim of stabilizer} and equations \eqref{eq: dim Hilbert - matrices} we now get
$$\dim H_\h= 4+r_k-3=\frac{(k+2)(k+1)}{2}+1,$$
which is much bigger than $k+2$ if $k$ is large.
\end{ex}

\end{document}